\documentclass[a4paper,12pt]{amsart}
\calclayout
\usepackage{amssymb}
\usepackage{amsmath}
\usepackage{amsfonts}
\usepackage[all]{xy}
\usepackage{mathabx}
\usepackage[dvipdfm,CJKbookmarks,bookmarks=true,colorlinks=false]{hyperref}

\newtheorem{thm}{Theorem}[section]

\newtheorem{prop}[thm]{Proposition}
\theoremstyle{definition}
\newtheorem{defn}[thm]{Definition}
\theoremstyle{remark}

\numberwithin{equation}{section}

\begin{document}

\title[Tensor products of function systems revisited]{Tensor products of function systems revisited}

\author{Kyung Hoon Han}

\address{Department of Mathematics, The University of Suwon, Gyeonggi-do 445-743, Korea}

\email{kyunghoon.han@gmail.com}



\thanks{This work was supported by the National Research Foundation of Korea Grant funded by the Korean Government (2012R1A1A1012190)}

\date{}

\dedicatory{}

\commby{}


\begin{abstract}
Based on the Archimedeanization developed by Paulsen and Tomforde, we give an explicit description for the positive cones of  maximal tensor products of function systems. From this description, we obtain an approximation theorem for nuclear maps between function systems. As an application, we give elementary proofs on several characterizations of nuclear function systems that are already known.
\end{abstract}

\maketitle

\section{Introduction}

A real vector space $V$ is called an ordered real vector space if there exists a cone $V^+ \subset V$ such that $V^+ \cap - V^+ = \{ 0\}$. The cone $V^+$ induces a partial order by $v \ge w$ if and only if $v-w \in V^+$. For an ordered real vector space $(V, V^+)$, an element $e$ in $V$ is called an order unit if for each $v$ in $V$, there exists a real number $r>0$ such that $r e \ge v$. We call an order unit $e$ an Archimedean order unit if $\varepsilon e + v \in V^+$ for any $\varepsilon >0$ implies that $v \in V^+$. The order norm of an ordered real vector space with an Archimedean order unit is defined by $$\| v \| = \inf \{ r>0 : -re \le v \le re \}.$$ An ordered real vector space equipped with an Archimedean order unit is called an Archimedean (partially) ordered (vector) space \cite{Ka,PT} or a function system \cite{Ef,NP}. Kadison proved that every Archimedean ordered space can be embedded into a real continuous function algebra on a compact Hausdorff space through a unital order isomorphism that is also isometric with respect to the order norm \cite{Ka}. Kadison's representation theorem for an Archimedean ordered space justifies the alternative term {\sl function system}.

The tensor products of function systems have been studied in \cite{NP} and \cite{Ef}. Two extremal tensor products of function systems, the minimal tensor product and the maximal tensor product were considered. A function system $V$ is called nuclear if $V \otimes_{\min} W = V \otimes_{\max} W$ for any function system $W$. Along the line of the duality between function systems and compact convex sets, Namioka and Phelps proved that the nuclearity corresponds to the Choquet simplex \cite[Theorem 2.2]{NP}.

Recently, Paulsen and Tomforde introduced the Archimedeanization \cite{PT} and this has been applied to operator system theory in a series of papers \cite{PTT,KPTT1,KPTT2}. In this paper, we focus on its application to function systems. Section 3 gives an explicit description for the positive cones of  maximal tensor products of function systems applying the Archimedeanization. Based on this, we prove that the maximal tensor product is projective as a bifunctor on the category consisting of function systems and unital positive maps. Section 4 applies the ideas of \cite{HP} to function systems and obtains an approximation theorem for nuclear maps between function systems. As an application, we give elementary proofs for several characterizations of nuclear function systems that are already known \cite[Theorem 7.4]{Ef}, \cite[Theorem 1.4]{NP}.

The referee kindly pointed out to the author that there has been another line of research on tensor products of Archimedean partially ordered vector spaces with order unit and Archimedeanization procedures, mainly influenced by Fremlin and extended by others, for which we refer to \cite{El,F,GL,BR,GK}.

\section{Preliminaries}

Two extremal tensor products of function systems have been studied in \cite{NP} and \cite[Section 7]{Ef}. Suppose that $V$ and $W$ are function systems. Given faithful representations $\varphi : V \to C(X)$ and $\psi : W \to C(Y)$, their tensor product $\varphi \otimes \psi : V \otimes W \to C(X \times Y)$ is also faithful \cite[Section 7]{Ef}. The minimal tensor product $V \otimes_{\min} W$ is defined as the function system structure on $V \otimes W$ induced by $\varphi \otimes \psi$. The minimal tensor product is independent of the choice of the faithful representations $\varphi$ and $\psi$ \cite[Lemma 7.1]{Ef}. The minimal tensor products of function systems have an alternative description. The positive cone of the minimal tensor product $V \otimes_{\min} W$ is given by $$(V \otimes_{\min} W)^+ = \{ z \in V \otimes W : (f \otimes g) (z) \ge 0, f \in S(V), g \in S(W) \},$$ where $S(V)$ denotes the state space on $V$ \cite{NP}.

Let $(V \otimes W)^d$ denote the algebraic dual of $V \otimes W$. The maximal state space is defined as $$S_{\max} (V \otimes W) := \{ r \in (V \otimes W)^d : r|_{V^+ \otimes W^+} \ge 0 \}$$ which is a weak$^d$ compact convex subset of $(V \otimes W)^d$. We can regard elements of $V \otimes W$ as continuous affine functions on $S_{\max}(V \otimes W)$. This is a faithful realization. The maximal tensor product $V \otimes_{\max} W$ is defined as the function system structure on $V \otimes W$ induced by the inclusion $V \otimes W \subset C(S_{\max}(V \otimes W))$ \cite[Section 7]{Ef}.

If $\varphi : V_1 \to V_2$ and $\psi : W_1 \to W_2$ are unital positive maps, then their tensor products $\varphi \otimes \psi : V_1 \otimes_{\min} W_1 \to V_2 \otimes_{\min} W_2$ and $\varphi \otimes \psi : V_1 \otimes_{\max} W_1 \to V_2 \otimes_{\max} W_2$ are unital positive. In particular, $\varphi \otimes \psi : V_1 \otimes_{\min} W_1 \to V_2 \otimes_{\min} W_2$ is an order embedding if $\varphi$ and $\psi$ are order embeddings. This is not true for the maximal tensor product.

There is a natural duality between compact convex sets and function systems. For a compact convex set $K$, the space $A(K)$ of real continuous affine functions on $K$ is a function system. For a function system $V$, the state space $S(V)$ equipped with a weak$^*$-topology is a compact convex set. By the barycenter formula \cite[Proposition 1.2.2]{A}, the state space $S(A(K))$ consists entirely of evaluations at points in $K$. In fact, the state space $S(A(K))$ is affinely homeomorphic to $K$. Conversely, the continuous affine functions $A(S(V))$ consists entirely of evaluations at elements in $V$. The affine function system $A(S(V))$ is unitally order isomorphic to $V$.

Let $P(K)$ denote the cone of real continuous convex functions on a compact convex set $K$. For $\mu, \nu \in M_{\mathbb R}(K)$, the Choquet order is defined as $$\mu \prec \nu \quad \Leftrightarrow \quad \mu(f) \le \nu (f), \ \forall f \in P(K).$$ Roughly speaking, the Choquet order measures how far the mass of a measure is distributed to the outside. A complex measure $\mu$ on a compact convex set $K$ is said to be a boundary measure if $|\mu|$ is a maximal element of $M^+(K)$ with respect to the Choquet order. Every point $x$ in a compact convex set $K$ can be represented by a positive normalized boundary measure $\mu$; that is, $$a(x)=\int_K a(y) d\mu (y), \qquad \forall a \in A(K).$$ The boundary measure associated with each point in $K$ is unique if and only if the dual space $A(K)^*$ is lattice ordered. In this case, we call $K$ a Choquet simplex.

A function system $V$ is called nuclear if $V \otimes_{\min} W = V \otimes_{\max} W$ for any function system $W$. Namioka and Phelps proved that a compact convex set $K$ is a Choquet simplex if and only if the continuous affine function system $A(K)$ on $K$ is nuclear. Dually, a function system $V$ is nuclear if and only if its state space $S(V)$ is a Choquet simplex.

Paulsen and Tomforde introduced a functorial process, called Archimedeanizaton, for forming an Archimedean ordered space from an ordered real vector space with an order unit. Given an ordered real vector space $(V,V^+)$ with an order unit $e$, we let $$D := \{ v \in V : \varepsilon e + v \in V^+\ \text{for all}\ \varepsilon>0 \} \quad \text{and} \quad N := D \cap -D.$$ The Archimedeanization $V_{\rm Arch}$ of $V$ is defined as an ordered real vector space $(V \slash N, D+N)$ with an order unit $e+N$. Then, $V_{\rm Arch}$ is an Archimedean ordered space. The Archimedeanization is characterized by the universal property that it satisfies: for an Archimedean ordered space $W$ and a
unital positive map $\varphi : V \to W$, there exists a unique unital positive linear map $\tilde \varphi : V_{\rm Arch} \to W$ with $\varphi = \tilde \varphi \circ q$. $$\xymatrix{V \ar[rr]^{q} \ar[dr]_{\varphi} && V_{\rm Arch} \ar[dl]^{\tilde \varphi} \\ &W &}$$

We say that a subspace $J$ of an Archimedean ordered space $V$ is an order ideal of $V$ if $p \in J$ and $0 \le q \le p$ imply that $q \in J$. The Archimedean quotient of $V$ by $J$ is defined as the Archimedeanization of $(V \slash J, V^+ + J, e+J)$. Given Archimedean ordered spaces $V,W$ and a unital positive linear map $\varphi : V \to W$, the Archimedean quotient by $\ker \varphi$ is unitally order isomorphic to $(V \slash \ker \varphi, (V \slash \ker \varphi )^+, e+\ker \varphi)$, where $$(V \slash \ker \varphi )^+ := \{ v + \ker \varphi : \forall \varepsilon>0, \exists j \in \ker \varphi \ \text{such that}\ j+\varepsilon e + v \in V^+ \}.$$ The map $\tilde \varphi : V \slash \ker \varphi \to W$ given by $\tilde \varphi (v + \ker \varphi ) = \varphi (v)$ is a unital positive linear map. $$\xymatrix{V \ar[rr]^{q} \ar[dr]_{\varphi} && V \slash \ker \varphi \ar[dl]^{\tilde \varphi} \\ & W &}$$

\section{Maximal tensor products of function systems}

First, we give an explicit description for the positive elements in the maximal tensor products of function systems applying Archimedeanization. We define $$V^+ \otimes W^+ := \{ \sum_{i=1}^n v_i \otimes w_i \in V \otimes W : n \in \mathbb N,  v_i \in V^+, w_i \in W^+ \}$$ and $$D := \{ z \in V \otimes W : \forall \varepsilon > 0, z + \varepsilon e_V \otimes e_W \in V^+ \otimes W^+ \}.$$ For $v \in V$ and $w \in W$, we have $$\begin{aligned} & \|v\| \|w\| e_V \otimes e_W \pm v \otimes w \\ = & {1 \over 2}(\|v\| e_V \pm v) \otimes (\|w\| e_W + w) + {1 \over 2} (\|v\| e_V \mp v) \otimes (\|w\| e_W - w) \in V^+ \otimes W^+.\end{aligned}$$ Hence, $(V \otimes W, V^+ \otimes W^+, e_V \otimes e_W)$ is an ordered real vector space with an order unit. Given $z \in D$ and $f \in S(V), g \in S(W)$, we have $$0 \le (f \otimes g)(z+\varepsilon e_V \otimes e_W) \le (f \otimes g)(z)+\varepsilon$$ for any $\varepsilon>0$. It follows that $D \subset (V \otimes_{\min} W)^+$, so $N:=D \cap -D = \{ 0\}$. Hence, the triple $(V \otimes W, D, e_V \otimes e_W)$ is an Archimedeanization of $(V \otimes W, V^+ \otimes W^+, e_V \otimes e_W)$.

\begin{thm}
Suppose that $(V,V^+,e_V)$ and $(W,W^+,e_W)$ are function systems. Then the maximal tensor product $V \otimes_{\max} W$ coincides with the Archimedeanization $(V \otimes W, D, e_V \otimes e_W)$ of $(V \otimes W, V^+ \otimes W^+, e_V \otimes e_W)$.
\end{thm}

\begin{proof}
$\subset$) Let $z \in (V \otimes_{\max} W)^+$ and $f$ be a state on $(V \otimes W, D, e_V \otimes e_W)$. Since $V^+ \otimes W^+ \subset D$, we have $f|_{V^+ \otimes W^+} \ge 0$, so $f \in S_{\max}(V \otimes W)$. It follows that $f(z) \ge 0$ for all states $f$ on $(V \otimes W, D, e_V \otimes e_W)$. By \cite[Proposition 2.20]{PT}, $z$ belongs to $D$.

$\supset$) Let $z \in D$ and $f \in S_{\max} (V \otimes W)$. Since $z+\varepsilon e_V \otimes e_W \in V^+ \otimes W^+$ for $\varepsilon>0$, we have $$0 \le f(z+\varepsilon e_V \otimes e_W) = f(z) + \varepsilon.$$ It follows that $f(z) \ge 0$, so $z \in (V \otimes_{\max} W)^+$.
\end{proof}

The maximal tensor products of function systems are characterized by the following universal property.

\begin{prop}\label{universal}
Suppose that $V, W$ and $Z$ are function systems and $\Phi : V \times W \to Z$ is a bilinear map such that $\Phi (v,w) \in Z^+$ for all $v \in V^+$ and $w \in W^+$. Then
there exists a unique positive linear map $\tilde \Phi : V \otimes_{\max} W \to Z$ such that $\Phi (v,w) = \tilde \Phi (v \otimes w)$. $$\xymatrix{V \times W \ar[rr] \ar[dr]_{\Phi} && V \otimes_{\max} W \ar[dl]^{\tilde \Phi} \\ &Z &}$$
\end{prop}

\begin{proof}
Suppose that $z \in (V \otimes_{\max} W)^+$. Then, we have $z+\varepsilon e_V \otimes e_W \in V^+ \otimes W^+$ for any $\varepsilon > 0$. It follows that $$0 \le \tilde{\Phi}(z+\varepsilon e_V \otimes e_W) = \tilde{\Phi} (z) + \varepsilon \Phi(e_V, e_W) \le \tilde{\Phi}(z) + \varepsilon \|\Phi(e_V, e_W)\| e_Z$$ for any $\varepsilon > 0$. Since $e_Z$ is Archimedean, $\tilde \Phi(z)$ belongs to $Z^+$.
\end{proof}

\begin{prop}\label{duality1}
Suppose that $V$ and $W$ are function systems. A functional $\varphi : V \otimes_{\max} W \to \mathbb R$ is positive if and only if its associated linear map $L_\varphi : V \to W^*$ is positive.
\end{prop}

\begin{proof}
Since $$\langle L_\varphi (v), w \rangle = \varphi (v \otimes w), \qquad v \in V, w \in W,$$ $L_\varphi : V \to W^*$ is positive if and only if $\varphi : V \times W \to \mathbb R$ is a positive bilinear map. By Proposition \ref{universal}, this is equivalent to the positivity of $\varphi : V \otimes_{\max} W \to \mathbb R$.
\end{proof}

The second dual of a function system is also a function system \cite[Proposition 3.7]{Ef}.

\begin{prop}\label{bidual}
Suppose that $V, W$ are function systems and $\iota_{W} : W \to W^{**}$ is a canonical inclusion. Then the map $${\rm id}_V \otimes \iota_W : V \otimes_{\max} W \to V \otimes_{\max} W^{**}$$ is an order embedding.
\end{prop}

\begin{proof}
Suppose that ${\rm id}_V \otimes \iota_W (z) \in (V \otimes_{\max} W^{**})^+$ and $\varphi \in S(V \otimes_{\max} W)$. By Proposition \ref{duality1}, the associated linear map $L_\varphi : V \to W^*$ is positive. Since the composition $\iota_{W^*} \circ L_\varphi : V \to W^{***}$ is also positive, its associated functional $\varphi^\wedge$ on $V \otimes_{\max} W^{**}$ is a state by Proposition \ref{duality1} again. From $$\varphi (z) = \varphi^\wedge ({\rm id}_V \otimes \iota_W (z)) \ge 0, \qquad \varphi \in S(V \otimes_{\max} W),$$ we see that $z \in (V \otimes_{\max} W)^+$.
\end{proof}

\begin{prop}\label{cross}
Suppose that $(V,V^+,e_V)$ and $(W,W^+,e_W)$ are function systems. The order norms $\| \cdot \|_{V
\otimes_{\min} W}$ and $\| \cdot \|_{V \otimes_{\max} W}$ are cross norms with respect to the order norms of $V$ and $W$. In addition, the inequality $\| \cdot \|_{V \otimes_{\min} W} \le \| \cdot \|_{V \otimes_{\max} W}$ holds.
\end{prop}

\begin{proof}
Let $v \in V$ and $w \in W$. From $$\begin{aligned} & \|v\| \|w\| e_V \otimes e_W \pm v \otimes w \\
= & {1 \over 2}(\|v\| e_V \pm v) \otimes (\|w\| e_W + w) + {1 \over 2} (\|v\| e_V \mp v) \otimes (\|w\| e_W - w) \in V^+ \otimes W^+,\end{aligned}$$ we see that $\| \cdot \|_{V \otimes_{\max} W}$ is a subcross norm. By the definition of the order norm, the inclusion $(V \otimes_{\max} W)^+ \subset (V \otimes_{\min} W)^+$ implies the inequality $\| \cdot \|_{V \otimes_{\min} W} \le \| \cdot \|_{V \otimes_{\max} W}$. It follows that $$\begin{aligned} \|v\|\|w\| & = \sup \{ |(f \otimes g)(v \otimes w)| : f \in S(V), g \in S(W) \} \\ & \le \| v \otimes w \|_{V \otimes_{\min} W} \\ & \le \| v \otimes w \|_{V \otimes_{\max} W} \\ & \le \|v\|\|w\|,\end{aligned}$$ because $f \otimes g$ is a state on $V \otimes_{\min} W$.
\end{proof}

Let $A \in \mathbb M_m (\mathbb R)$ and $B \in \mathbb M_n (\mathbb R)$. We denote the functional $$X \in M_m(\mathbb R) \mapsto {\rm tr} (XA) \in \mathbb R$$ by ${\rm tr} (\ \cdot\ A)$. From $$A^*A \otimes B^*B = (A \otimes B)^* (A \otimes B) \qquad \text{and} \qquad {\rm tr} (\ \cdot\ A) \otimes {\rm tr} (\ \cdot\ B) = {\rm tr} (\ \cdot\ (A \otimes B)),$$ we see that $$(\mathbb M_m(\mathbb R) \otimes_{\max} \mathbb M_n(\mathbb R))^+ \subset \mathbb M_{mn}(\mathbb R)^+ \subset (\mathbb M_m(\mathbb R) \otimes_{\min} \mathbb M_n(\mathbb R))^+.$$ Since the transpose map $\bf t$ on $\mathbb M_2(\mathbb R)$ is a unital positive map and $${\rm id}_{\mathbb M_2(\mathbb R)} \otimes {\bf t} (\begin{pmatrix} 1&0&0&1 \\ 0&0&0&0 \\ 0&0&0&0 \\ 1&0&0&1 \end{pmatrix}) = \begin{pmatrix} 1&0&0&0 \\ 0&0&1&0 \\
0&1&0&0 \\ 0&0&0&1 \end{pmatrix} \notin \mathbb M_4(\mathbb R)^+,$$ we see that $$\begin{pmatrix} 1&0&0&1 \\ 0&0&0&0 \\ 0&0&0&0 \\ 1&0&0&1 \end{pmatrix} \in \mathbb M_4(\mathbb R)^+\ \backslash\ (\mathbb M_2(\mathbb R) \otimes_{\max} \mathbb M_2(\mathbb R))^+$$ and  $$\begin{pmatrix} 1&0&0&0 \\ 0&0&1&0 \\ 0&1&0&0 \\ 0&0&0&1 \end{pmatrix} \in (\mathbb M_2(\mathbb R) \otimes_{\min} \mathbb M_2(\mathbb R))^+\ \backslash\ \mathbb M_4(\mathbb R)^+.$$

\begin{defn}
Suppose that $T : V \to W$ is a unital positive surjective linear map for function systems $V$ and $W$. We call $T : V \to W$ {\sl an order quotient map} if for any $w$ in $W^+$ and $\varepsilon>0$, we can find an element $v$ in $V$ such that $$v + \varepsilon e_V \in V^+ \quad \text{and} \quad T(v) = w.$$
\end{defn}

The key point of the above definition is that the lifting $v$ depends on the choice of $\varepsilon>0$. By slightly modifying \cite[Theorem 2.45]{PT}, we obtain the following proposition, which justifies the term {\sl order quotient map}.

\begin{prop}
Suppose that $T : V \to W$ is a unital positive surjective linear
map for function systems $V$ and $W$. Then $T : V \to W$ is an
order quotient map if and only if $\tilde{T} : V \slash \ker T \to
W$ is an order isomorphism.
\end{prop}

\begin{proof}
$T : V \to W$ is an order quotient map \begin{enumerate} \item[$\Leftrightarrow$] $\forall w \in W^+, \forall \varepsilon>0, \exists v \in V, v+ \varepsilon e_V \in V^+$ and $T(v)=w$ \item[$\Leftrightarrow$] $\forall w \in W^+, \exists v \in V, v + \ker T \in (V \slash \ker T)^+$ and $\tilde T(v + \ker T)=w$ \item[$\Leftrightarrow$] ${\tilde T} : V \slash \ker T \to W$ is an order isomorphism.
\end{enumerate}
\end{proof}

Recall that a bounded linear map $T : V \to W$ for normed spaces $V$ and $W$ is called a quotient map if it maps the open unit ball of $V$ onto the open unit ball of $W$.

\begin{prop} \label{order norm}
Suppose that $T : V \to W$ is a unital positive linear map for function systems $V$ and $W$. If $T : V \to W$ is a quotient map, then it is an order quotient map.
\end{prop}

\begin{proof}
Let $w \in W^+$ and $\varepsilon >0$. Then we have $$-{1 \over 2} \|w\| e_W \le w - {1 \over 2}\|w\| e_W \le {1 \over 2} \|w\| e_W.$$ There exists an element $v$ in $V$ such that $$T(v)=w - {1 \over 2} \|w\| e_W \quad \text{and} \quad \|v\| \le {1 \over 2} \|w\| + \varepsilon.$$ It follows that $$v + {1 \over 2} \|w\| e_V + \varepsilon e_V \in V^+ \quad \text{and} \quad T(v+{1 \over 2} \|w\| e_V) = w.$$
\end{proof}

The following theorem shows the projectivity of the maximal tensor product.

\begin{thm}
For function systems $V_1, V_2, W$ and an order quotient map $Q : V_1 \to V_2$, the linear map $Q \otimes {\rm id}_W : V_1 \otimes_{\max} W \to V_2 \otimes_{\max} W$ is an order quotient map.
\end{thm}

\begin{proof}
Let $z \in (V_2 \otimes_{\max} W)^+$ and $\varepsilon > 0$. Then we have $$z + {\varepsilon \over 2} e_{V_2} \otimes e_W \in V_2^+ \otimes W^+.$$ We write $$z=\sum_{i=1}^n v_i \otimes w_i -{\varepsilon \over 2} e_{V_2} \otimes e_W$$ for $v_i \in V_2^+$ and $w_i \in W^+$. There exists $u_i$ in $V_1$ such that $$Q(u_i)=v_i \quad \text{and} \quad u_i + {\varepsilon \over 2n\|w_i\|} e_{V_1} \in V_1^+$$ for each $1 \le i \le n$. It follows that $$(Q \otimes {\rm id}_W)(\sum_{i=1}^n u_i \otimes w_i -{\varepsilon \over 2} e_{V_1} \otimes e_W) = \sum_{i=1}^n v_i \otimes w_i - {\varepsilon \over 2} e_{V_2} \otimes e_W = z$$ and $$\begin{aligned} &(\sum_{i=1}^n u_i \otimes w_i - {\varepsilon \over 2} e_{V_1} \otimes e_W) + \varepsilon e_{V_1} \otimes e_W \\ = & \sum_{i=1}^n (u_i + {\varepsilon \over 2n\|w_i\|}e_{V_1}) \otimes w_i + \sum_{i=1}^n {\varepsilon \over 2n} e_{V_1} \otimes (e_W - {1 \over \|w_i\|} w_i) \\ \in &\ V_1^+ \otimes W^+\end{aligned}$$
\end{proof}

\section{Nuclear function systems}

A function system $V$ is called {\sl nuclear} if the identity $$V \otimes_{\min} W = V \otimes_{\max} W$$ holds for any function system $W$.

\begin{prop}\label{dual}
Suppose that $E$ is a finite-dimensional function system and $s$ is a faithful state on $E$. Then, $(E^*, (E^*)^+, s)$ is a function system.
\end{prop}

\begin{proof}
Let $f$ be a linear functional on $E$. Since the set $K :=\{ x \in E^+ : \|v\|=1 \}$ is compact and $s(x) > 0$ for any $x \in K$, the continuous function $|{f \over s} |$ has a maximum $M$ on $K$. Then we have $-M s \le f \le M s$. Suppose that $f+\varepsilon s \ge 0$ for any $\varepsilon>0$. Then $$0 \le (f + \varepsilon s)(x) = f(x) + \varepsilon s(x)$$ for all $\varepsilon >0$ and $x \in E^+$, so $f \ge 0$.
\end{proof}

\begin{prop}\label{duality2}
Suppose that $E$ and $W$ are function systems with $E$ finite dimensional. Then $z \in (E^* \otimes_{\min} W)^+$ if and only if its associated linear map $T_z : E \to W$ is positive.
\end{prop}

\begin{proof}
Since $E$ is finite dimensional, every positive functional on $E^*$ is an evaluation $\hat{v}$ for some $v \in E^+$. We have $$z \in (E^* \otimes_{\min} W)^+ \qquad \text{if and only if} \qquad 0 \le (\hat{v} \otimes f) (z) = f(T_z(v))$$ for all $v \in E^+$ and $f \in (W^*)^+$, which is equivalent to the positivity of $T_z : E \to W$.
\end{proof}

We apply the idea of \cite[Theorem 3.1]{HP} to function systems.

\begin{thm}\label{main1}
Suppose that $\Phi : V \to W$ is a unital positive map for function systems $V$ and $W$. The following are equivalent:
\begin{enumerate}
\item[(i)] the map $${\rm id}_A \otimes \Phi : A \otimes_{\min} V \to A \otimes_{\max} W$$ is positive for any function system $A$; \item[(ii)] the map $${\rm id}_E \otimes \Phi : E \otimes_{\min} V \to E \otimes_{\max} W$$ is positive for any finite dimensional function system $E$; \item[(iii)] there exist nets of unital positive maps $\varphi_\lambda : V \to \ell^\infty_{n_\lambda}$ and $\psi_\lambda : \ell^\infty_{n_\lambda} \to W$ such that $\psi_\lambda \circ \varphi_\lambda$ converges to the map $\Phi$ in the point-norm topology.
$$\xymatrix{V \ar[rr]^\Phi \ar[dr]_{\varphi_\lambda} && W \\ & \ell^\infty_{n_\lambda} \ar[ur]_{\psi_\lambda} &}$$
\end{enumerate}
\end{thm}

\begin{proof}
Clearly, (i) implies (ii).

$\rm (iii)\Rightarrow (i).$  First, let us show that $\ell^\infty_n$ is nuclear. We choose an element $\sum_{i=1}^n x_i \otimes e_i$ in $(A \otimes_{\min} \ell_n^\infty)^+$. We have $$0 \le (f \otimes e_j') (\sum_{i=1}^n x_i \otimes e_i) = f(x_j)$$ for all $f \in S(A)$ and $1 \le j \le n$. We see that $x_j \in A^+$, thus $$\sum_{i=1}^n x_i \otimes e_i \in A^+ \otimes {\ell^\infty_n}^+.$$ Hence, $\ell^\infty_n$ is nuclear. The map $${\rm id}_A \otimes \psi_\lambda \circ \varphi_\lambda : A \otimes_{\min} V \to A \otimes_{\min} \ell^\infty_{n_\lambda} = A \otimes_{\max} \ell^\infty_{n_\lambda} \to A \otimes_{\max} W$$ is unital positive.  Since $\|\cdot \|_{A \otimes_{\max} W}$ is a cross norm by Proposition \ref{cross}, $({\rm id}_A \otimes \psi_\lambda \circ \varphi_\lambda) (z)$ converges to ${\rm id}_A \otimes \Phi (z)$ for each $z \in A \otimes V$. It follows that $z \in (A \otimes_{\min} V)^+$ implies ${\rm id}_A \otimes \Phi (z) \in (A \otimes_{\max} W)^+$.

\vskip 1pc

$\rm (ii) \Rightarrow (iii).$ Let $E$ be a finite dimensional function subsystem of $V$ and $s$ be a faithful state on $E$. By Proposition \ref{dual}, $(E^*, (E^*)^+, s)$ is a function system. By Proposition \ref{duality2}, we can regard the inclusion $\iota : E \subset V$ as an element in $(E^* \otimes_{\min} V)^+$. The restriction $\Phi |_E : E \to W$ can be identified with an element $({\rm id}_{E^*} \otimes \Phi) (\iota)$. By assumption, this belongs to $(E^* \otimes_{\max} W)^+$. We consider the directed set $$\Omega = \{(E, \varepsilon) : \text{$E$ is a finite dimensional
function subsystem of $V$}, \varepsilon >0 \}$$ with the standard partial order. Let $\lambda = (E, \varepsilon)$. For any $\varepsilon >0$, the restriction $\Phi|_E$ can be written as $$\Phi|_E + \varepsilon s \otimes e_W = \sum_{k=1}^{n_\lambda} f_k \otimes w_k$$ for nonzero $f_k \in (E^*)^+$ and $w_k \in W^+$ ($1 \le k \le n_\lambda$). The map $$f : x \in E \mapsto (f_1(x), \cdots, f_{n_\lambda}(x)) \in \ell^\infty_{n_\lambda}$$ is positive, and we may assume that $f(e_V) = (c_1, \cdots, c_{m_\lambda}, 0, \cdots, 0)$ for $c_k>0$ by rearrangement. For $x \in E^+$, we have $$0 \le f(x) \le \|x\|f(e_V) = \|x\| (c_1, \cdots, c_{m_\lambda}, 0, \cdots, 0).$$ Since every element in $E$ can be written as a difference of positive elements in $E$ and each $f_k$ is assumed to be nonzero, we have that $m_\lambda = n_\lambda$ and $f(e_V)$ is invertible. Because $$\sum_{k=1}^{n_\lambda} f_k \otimes w_k = \sum_{k=1}^{n_\lambda} {f_k \over f_k (e_V)} \otimes f_k (e_V) w_k,$$ we may assume that $f$ is a unital positive map. By the Krein theorem, $f : E \to \ell^\infty_{n_\lambda}$ extends to a unital positive map $\varphi_\lambda : V \to \ell^\infty_{n_\lambda}$. We define a positive map $\psi'_\lambda : \ell^\infty_{n_\lambda} \to W$ by $$\psi'_\lambda(c_1,\cdots,c_{n_\lambda}) = \sum_{k=1}^{n_\lambda} c_k w_k.$$ For $x \in E$, we have $$\|\Phi(x)-\psi'_\lambda \circ \varphi_\lambda (x) \| = \|\Phi(x)-\sum_{k=1}^{n_\lambda} f_k(x) w_k \| = \varepsilon \| s(x) e_W \| \le \varepsilon \|x\|.$$ Hence, we can take nets of unital positive maps $\varphi_\lambda : V \to \ell^\infty_n$ and positive maps $\psi'_\lambda :
\ell^\infty_{n_\lambda} \to W$ such that $\psi'_\lambda \circ \varphi_\lambda$ converges to the map $\Phi$ in the point-norm topology. Since each $\varphi_\lambda$ is unital, $\psi'_\lambda (1_{n_\lambda})$ converges to $e_W$. Let us choose a state $\omega_\lambda$ on $\ell^\infty_{n_\lambda}$ and set $$\psi_\lambda(c) := {1 \over \| \psi'_\lambda\|} \psi'_\lambda(c) + \omega_\lambda(c) (e_W - {1 \over \| \psi'_\lambda \|} \psi'_\lambda (1_{n_{\lambda}})).$$ Then $\psi_\lambda : \ell^\infty_{n_\lambda} \to W$ is a unital positive map such that $\psi_\lambda \circ \varphi_\lambda$ converges to the map $\Phi$ in the point-norm topology.
\end{proof}

As an application of Theorem \ref{main1}, we give elementary proofs of known characterizations of nuclear function systems \cite[Theorem 7.4]{Ef}, \cite[Theorem 2.2]{NP} except for the implication from the last.

\begin{thm}
Let $V$ be a function system. The following are equivalent:
\begin{enumerate}
\item[(i)] $V$ is nuclear; \item[(ii)] we have $$V \otimes_{\min} E = V \otimes_{\max} E$$ for any finite dimensional function system $E$; \item[(iii)] there exist nets of unital positive maps $\varphi_\lambda : V \to \ell^\infty_{n_\lambda}$ and $\psi_\lambda : \ell^\infty_{n_\lambda} \to V$ such that $\psi_\lambda \circ \varphi_\lambda$ converges to ${\rm id}_V$ in the point-norm topology; \item[(iv)] there exist nets of weak$^*$-continuous unital positive maps $\varphi_\lambda : V^{**} \to \ell^\infty_{n_\lambda}$ and $\psi_\lambda : \ell^\infty_{n_\lambda} \to V^{**}$ such that $\psi_\lambda \circ \varphi_\lambda$ converges to ${\rm id}_{V^{**}}$ in the point-weak$^*$ topology; \item[(v)] the second dual $V^{**}$ is injective; \item[(vi)] given $V \subset V^{**} \subset C(K)$ with a compact set $K$, there exists a unital positive map $\Phi : C(K) \to V^{**}$ such that $\Phi|_V = {\rm id}_V$; \item[(vii)] given both any function system $W_1$ containing $V$ and a function system $W_2$, the inclusion $V \otimes_{\max} W_2 \subset W_1 \otimes_{\max} W_2$ is an order embedding; \item[(viii)] given both any function system $W$ containing $V$ and a finite dimensional function system $E$, the inclusion $V \otimes_{\max} E \subset W \otimes_{\max} E$ is an order embedding; \item[(ix)] given both any function systems $W_1 \subset W_2$, the inclusion $V \otimes_{\max} W_1 \subset V \otimes_{\max} W_2$ is an order embedding; \item[(x)] the state space $S(V)$ is a Choquet simplex.
\end{enumerate}
\end{thm}

\begin{proof}
Putting $V=W$ and $\Phi = {\rm id}_V$ in Theorem \ref{main1}, we obtain the equivalences $\rm (i)\Leftrightarrow (ii) \Leftrightarrow (iii)$.  For a compact set $K$, we can show that $C(K)$ satisfies (iii) using the partition of unity. For the details, see the proof of \cite[Theorem 2.3.7]{L}. The implications $\rm (v)\Rightarrow (vi)$, $\rm (vii) \Rightarrow (viii)$ and $\rm (i) \Rightarrow (ix)$ are trivial.

$\rm (iii)\Rightarrow (iv).$ The proof consists of the idea of \cite[Thoerem 4.5]{EOR} and the perturbation argument. Take $\varepsilon>0$ and a finite dimensional function subsystem $E$ of $V^{**}$ and a finite dimensional subspace $F$ of $V^*$. By the principle of local reflexivity, we can choose a linear map $\theta : E \to V$ such that $\|\theta\| < 1+\varepsilon$, $\theta (v) = v$ for $v \in E \cap V$ and $\langle \theta (x), f \rangle = \langle x, f \rangle$ for $x \in E, f \in F$. By the second condition, the map $\theta : E \to V$ is unital. By assumption, there exist unital positive maps $\varphi : V \to \ell^\infty_n$ and $\psi : \ell^\infty_n \to V$ such that $\|\psi \circ \varphi (\theta (x)) -\theta (x)\| \le \varepsilon \|\theta (x)\|$ for all $x \in E$. Since the inclusion $\iota : E \subset V^{**}$ is the dual map of $\iota^*|_{V^*} : V^* \to E^*$, the map $\iota^*|_{V^*} : V^* \to E^*$ is a quotient map. Applying this to each coordinate of $\varphi \circ \theta : E \to \ell^\infty_n$, we obtain its weak$^*$-continuous extension $\varphi' : V^{**} \to \ell^\infty_n$ with $\|\varphi'\|<1+\varepsilon$. $$\xymatrix{E \ar@{^{(}->}[r] \ar[dd]_\theta& V^{**} \ar[r] \ar[d]_{\varphi'} & V^{**} \\ & \ell^\infty_n \ar[rd]^\psi & \\  V \ar[ru]^\varphi \ar[rr]^{{\rm id}_V} && V \ar@{^{(}->}[uu]}$$ Suppose that $V \subset C(K)$ for a compact set $K$. By the Hahn-Banach theorem and the Riesz representation theorem, $\varphi'|_V : V \to \ell^\infty_n$  is the restriction of $(\mu_1, \cdots, \mu_n) : C(K) \to \ell^\infty_n$ for some $\mu_k \in M_{\mathbb R}(K)$ with $\|\mu_k\|<1+\varepsilon$. $$\xymatrix{C(K) \ar[drr]^-{(\mu_1, \cdots, \mu_n)} && \\ V \ar@{^{(}->}[u] \ar[rr]_-{\varphi'|_V} && \ell^\infty_n}$$ Let $\mu_k = \mu_k^+ - \mu_k^-$ be the Jordan decomposition. Then, we have $$\mu_k^+(1) \ge \mu_k^+(1) - \mu_k^-(1) = \mu_k(1) = 1.$$ It follows that $$\begin{aligned} \|\mu_k - {\mu_k^+ \over \mu_k^+(1)}\| & = \|\mu_k^+ - \mu_k^- - {\mu_k^+ \over \mu_k^+(1)}\| \\  & \le \|\mu_k^-\| + {\mu_k^+(1) -1 \over \mu_k^+(1)} \|\mu_k^+\| \\ & = \|\mu_k^-\|+\|\mu_k^+\|-1 \\ & = \|\mu_k\|-1 \\ & < \varepsilon. \end{aligned}$$ Let  $\varphi'' : V^{**} \to \ell^\infty_n$ be the second dual of the restriction of $({1 \over \mu_1^+(1)}\mu_1^+, \cdots, {1 \over \mu_n^+(1)}\mu_n^+)$ on $V$. Then $\varphi'' : V^{**} \to \ell^\infty_n$ is a weak$^*$-continuous unital positive map satisfying $$\|\varphi' - \varphi''\|=\|\varphi'|_V-\varphi''|_V\| \le \max \{ \|\mu_k - {\mu_k^+ \over \mu_k^+(1)}\| : 1 \le k \le n \}<\varepsilon.$$ For all $x \in E$ and $f \in F$, we have $$\begin{aligned} & |\langle \psi \circ \varphi'' (x) -x, f \rangle| \\ \le & |\langle \psi \circ \varphi'' (x) - \psi \circ \varphi' (x), f \rangle| + |\langle \psi \circ \varphi \circ \theta (x) -\theta (x), f \rangle| + |\langle \theta (x) -x, f \rangle| \\ < & (2 \varepsilon+\varepsilon^2) \|x\| \|f\|. \end{aligned}$$ The index set $\{ (\varepsilon, E, F) \}$ is directed in such a way that $(\varepsilon_1, E_1, F_1) \le (\varepsilon_2, E_2, F_2)$ if and only if $\varepsilon_1 \ge \varepsilon_2$ and $E_1 \subset E_2, F_1 \subset F_2$.

$\rm (iv)\Rightarrow (v).$ Suppose that $W_1 \subset W_2$ are function systems and $\Phi : W_1 \to V^{**}$ is a unital positive map. Let $\Phi_\lambda : W_2 \to \ell^\infty_{n_\lambda}$ be a unital positive extension of $\varphi_\lambda \circ \Phi : W_1 \to \ell^\infty_{n_\lambda}$. Then the point-weak$^*$ cluster point of $\psi_\lambda \circ \Phi_\lambda : W_2 \to V^{**}$ is the unital positive extension of $\Phi : W_1 \to V^{**}$.
$$\xymatrix{W_2 \ar[rr]^-{\Phi_\lambda} && \ell^\infty_{n_\lambda} \ar[dr]^{\psi_\lambda} & \\ W_1 \ar@{^{(}->}[u] \ar[r]^\Phi & V^{**} \ar[rr]^-{{\rm id}_{V^*}} \ar[ru]^{\varphi_\lambda} && V^{**}}$$

$\rm (vi)\Rightarrow (i).$ Suppose that $V^{**} \subset C(K)$ for a compact set $K$. By Proposition \ref{bidual}, we have $$\xymatrix{V \otimes_{\min} W \ar[r] & C(K) \otimes_{\min} W = C(K) \otimes_{\max} W \ar[r]^-{\Phi \otimes {\rm id}_W} & V^{**} \otimes_{\max} W \\ && V \otimes_{\max} W. \ar@{^{(}->}[u]}$$

$\rm (v)\Rightarrow (vii).$ Since there exists a positive projection from $W_1^{**}$ onto $V^{**}$, the inclusion $V^{**} \otimes_{\max} W_2 \subset  W_1^{**} \otimes_{\max} W_2$ is an order embedding. By Proposition \ref{bidual}, we have $$\xymatrix{W_1 \otimes_{\max} W_2 \ar@{^{(}->}[r] & W_1^{**} \otimes_{\max} W_2 \\ V \otimes_{\max} W_2 \ar@{^{(}->}[r] & V^{**} \otimes_{\max} W_2. \ar@{^{(}->}[u]}$$

$\rm (viii)\Rightarrow (ii).$ Suppose that $V \subset C(K)$ for a compact set $K$. The conclusion follows from $$\xymatrix{V \otimes_{\max} E \ar@{^{(}->}[r] & C(K) \otimes_{\max} E \\ V \otimes_{\min} E \ar@{^{(}->}[r] & C(K) \otimes_{\min} E. \ar@{=}[u]}$$

$\rm (ix)\Rightarrow (i).$ Suppose that $W \subset C(K)$ for a compact set $K$. The conclusion follows from $$\xymatrix{V \otimes_{\max} W \ar@{^{(}->}[r] & V \otimes_{\max} C(K) \\ V \otimes_{\min} W \ar@{^{(}->}[r] & V \otimes_{\min} C(K). \ar@{=}[u]}$$

$\rm (iii)\Rightarrow (x).$ Suppose that $f$ and $g$ are linear functionals on $V$. The net $\varphi_\lambda^* \circ \psi_\lambda^* : V^* \to V^*$ converges to ${\rm id}_{V^*}$ in the point-weak$^*$ topology. $$\xymatrix{V^* \ar[rr]^{{\rm id}_{V^*}} \ar[dr]_{\psi_\lambda^*} && V^* \\ & \ell^1_{n_\lambda} \ar[ur]_{\varphi_\lambda^*} &}$$ The ordered space $\ell^1_{n_\lambda} = (\ell^\infty_{n_\lambda})^*$ is lattice ordered. Let $h \in V^*$ be a weak$^*$ cluster point of $\varphi_\lambda^* (\psi_\lambda^* (f) \vee \psi_\lambda^* (g))$. Applying $\varphi_\lambda^*$ to $\psi_\lambda^* (f), \psi_\lambda^* (g) \le \psi_\lambda^* (f) \vee \psi_\lambda^* (g)$ and taking limit, we see that $f,g \le h$. If $f, g \le k$, then $\psi_\lambda^* (f), \psi_\lambda^* (g) \le \psi_\lambda^* (k)$. Applying $\varphi_\lambda^*$ to $\psi_\lambda^* (f) \vee \psi_\lambda^* (g) \le \psi_\lambda^* (k)$ and taking limit, we also see that $h \le k$.

For $\rm (x)\Rightarrow (i)$, see \cite[Theorem 2.2]{NP}. Alternatively, see \cite[Lemma 6.3, Corollary 6.4]{Ef} for $\rm (x)\Rightarrow (v)$.
\end{proof}

\section{The complex case}

An ordered $*$-vector space $(V,V^+)$ is a pair consisting of a $*$-vector space $V$ and a proper cone $V^+$ in $V_h$. We may define a partial ordering on $V_h$ by defining $v \ge w$ if and only if $v-w \in V^+$. If $(V,V^+)$ is an ordered $*$-vector space, an element $e \in V_h$ is called an order unit for $V$ if, for all $v \in V_h$, there exists a real number $r>0$ such that $re + v \ge 0$. If $(V,V^+)$ is an ordered $*$-vector space with an order unit $e$, then we say that $e$ is an Archimedean order unit if $v \in V^+$ whenever $v \in V$ and $\varepsilon e + v \ge 0$ for any $\varepsilon >0$. In this case, we call the triple $(V,V^+,e)$ an Archimedean ordered $*$-vector space. Paulsen and Tomforde proved that every Archimedean ordered $*$-vector space can be embedded into a complex continuous function algebra on a compact Hausdorff space through a unital order isomorphism that is also isometric with respect to the minimal order norm \cite[Theorem 5.2]{PT}.

In this section, we consider the tensor products of Archimedean ordered $*$-vector spaces. For $*$-vector spaces $V$ and $W$, the involution on $V \otimes W$ is defined as $$(v \otimes w)^* = v^* \otimes w^*$$ for $v \in V$ and $w \in W$.

\begin{defn}
Suppose that $(V,V^+,e_V)$ and $(W,W^+,e_W)$ are Archimedean ordered $*$-vector spaces.
\begin{enumerate}
\item We define a minimal tensor product $V \otimes_{\min} W$ as $(V \otimes W, (V \otimes_{\min} W)^+, e_V \otimes e_W)$, where $(V \otimes_{\min} W)^+ =\{ z \in V \otimes W : (f \otimes g)(z) \ge 0\ \text{for all}\ f \in
S(V), g \in S(W)\}$. \item We define a maximal tensor product $V \otimes_{\max} W$ as $(V \otimes W, (V \otimes_{\max} W)^+, e_V \otimes e_W)$, where $(V \otimes_{\max} W)^+ = \{ z \in V \otimes W : z + \varepsilon e_V \otimes e_W \in V^+ \otimes W^+\ \text{for all}\ \varepsilon>0 \}$.
\end{enumerate}
\end{defn}

If a norm $\|\cdot\|$ on $V$ satisfies $\|v\|=\|v^*\|$ for all $v \in V$ and extends the order norm on $V_h$, then we call it an order norm on $V$. Two extremal order norms on Archimedean ordered $*$-vector spaces are defined in \cite[Definition 4.4]{PT} and \cite[Definition 4.6]{PT}. The minimal order norm $\|\cdot\|_m$ on $V$ is defined by $$\|v\|_m := \sup \{ |f(v)| : f : V \to \mathbb C \ \text{is a state} \}.$$ The maximal order norm $\|\cdot\|_M$ on $V$ is defined by $$\|v\|_M := \inf \{ \sum_{i=1}^n |\lambda_i| \|v_i\| : v=\sum_{i=1}^n \lambda_i v_i \ \text{with} \ v_i \in V_h \ \text{and} \ \lambda_i \in \mathbb C \}.$$

\begin{thm}
Suppose that $(V,V^+,e_V)$ and $(W,W^+,e_W)$ are Archimedean ordered $*$-vector spaces.
\begin{enumerate}
\item We have $(V \otimes_{\min} W)^+ = (V_h \otimes_{\min} W_h)^+$. Hence, the minimal tensor product $V \otimes_{\min} W$ is an Archimedean ordered $*$-vector space. \item We have $(V \otimes_{\max} W)^+ = (V_h \otimes_{\max} W_h)^+$. Hence, the maximal tensor product $V \otimes_{\max} W$ is an Archimedean ordered $*$-vector space. \item The minimal order norm induced by the minimal tensor product is a cross norm with respect to the minimal order norms of $V$ and $W$. \item The maximal order norm induced by the maximal tensor product is a subcross norm with respect to the maximal order norms of $V$ and $W$.
\end{enumerate}
\end{thm}

\begin{proof}
(1) First, let us show that $(V \otimes_{\min} W)^+ \subset V_h \otimes W_h$. Let $z=\sum_{k=1}^n v_k \otimes w_k \in (V \otimes_{\min} W)^+$. We may assume that each $v_k$ is Hermitian and $\{ v_k \}_{k=1}^n$ is $\mathbb R$-linearly independent. For $f \in S(V)$ and $g \in S(W)$, we have $$\sum_{k=1}^n f(v_k) g(w_k^*) = \sum_{k=1}^n f(v_k) \overline{g(w_k)} = \overline{(f \otimes g) (z)} = (f \otimes g) (z) = \sum_{k=1}^n f(v_k) g(w_k).$$ It follows that $$f(\sum_{k=1}^n g(w_k - w_k^*) v_k ) = 0$$ for all $f \in S(V)$. By \cite[Proposition 3.12]{PT}, we have $\sum_{k=1}^n g(w_k - w_k^*) v_k = 0$. We see that $$\sum_{k=1}^n ({\rm Re}g(w_k - w_k^*)) v_k = \sum_{k=1}^n {\rm Re}(g(w_k - w_k^*) v_k) = 0 = \sum_{k=1}^n {\rm Im} (g(w_k - w_k^*) v_k) = \sum_{k=1}^n ({\rm Im} g(w_k - w_k^*)) v_k.$$ Since $\{ v_k \}_{k=1}^n$ is $\mathbb R$-linearly independent, we have $g(w_k - w_k^*)=0$ for all $g \in S(W)$ and $1 \le k \le n$. By \cite[Proposition 3.12]{PT} again, $w_k$ is Hermitian.

 For a real functional $f : V_h \to \mathbb R$, the complexification $\tilde f : V \to \mathbb C$ is defined by $$\tilde f (v) = f({v + v^* \over 2}) + i f({v - v^* \over 2i})$$ \cite[Definition 3.9]{PT}. Then, $f : V_h \to \mathbb R$ is a state if and only if $\tilde f : V \to \mathbb C$ is a state \cite[Proposition 3.10]{PT}. Moreover, every state on $V$ is realized in this form \cite[Proposition 3.11]{PT}. It follows that $$(V \otimes_{\min} W)^+ = (V_h \otimes_{\min} W_h)^+.$$

\bigskip

(3) We denote by $\| \cdot \|_{V \otimes_{\min} W,m}$ the minimal order norm induced by the minimal tensor product $V
\otimes_{\min} W$. For $v \in V$ and $w \in W$, we have $$\|v\|_m \|w\|_m =  \sup \{ |(f \otimes g)(v \otimes w)| : f \in S(V), g \in S(W) \} \le  \|v \otimes w \|_{V \otimes_{\min} W, m}$$ because $f \otimes g \in S(V \otimes_{\min} W)$. For a state $F$ on $V \otimes_{\min} W$, we let $$F(v \otimes w) = e^{i\theta} |F(v \otimes w)| \quad \text{and} \quad u = e^{-i \theta} v.$$  Then, we have $$\begin{aligned} & |F(v \otimes w)| \\ = & F(u \otimes w) \\ = & F({1 \over 2} u \otimes w + {1 \over 2} u^* \otimes w^*) \\ = & F(({u+u^* \over 2}) \otimes ({w+w^* \over 2}) - ({u - u^* \over 2i}) \otimes ({w - w^* \over 2i})) \\ \le & \| ({u+u^* \over 2}) \otimes ({w+w^* \over 2}) - ({u - u^* \over 2i}) \otimes ({w - w^* \over 2i}) \|_{V_h \otimes_{\min} W_h} \\ = & \sup \{ |f({u+u^* \over 2}) g({w+w^* \over 2}) - f({u - u^* \over 2i}) g({w - w^* \over 2i})| : f \in S(V_h), g \in S(W_h)\} \\ \le & \sup \{ |(f({u+u^* \over 2}) + i f({u - u^* \over
2i}))(g({w+w^* \over 2}) +i  g({w - w^* \over 2i}))| : f \in S(V_h), g \in S(W_h)\} \\ = & \sup \{ |\tilde f (u) \tilde g(w)| : f \in S(V_h), g \in S(W_h) \} \\ = & \|u\|_m \|w\|_m \\ = & \|v\|_m \|w\|_m. \end{aligned}$$

\bigskip

(4) We denote by $\| \cdot \|_{V \otimes_{\max} W,M}$ the maximal order norm induced by the maximal tensor product $V \otimes_{\max} W$. For $v \in V$ and $w \in W$, we write $$v = \sum_{k=1}^m \lambda_k v_k \quad \text{and} \quad w = \sum_{l=1}^n \mu_l w_l$$ for $\lambda_k, \mu_l \in \mathbb C$ and $v_k \in V_h, w_l \in W_h$. Then we have $$v \otimes w = \sum_{1 \le k \le m \atop 1 \le l \le n} \lambda_k \mu_l\ v_k \otimes w_l$$ and $$(\sum_{k=1}^m |\lambda_k| \|v_k\|)(\sum_{l=1}^n |\mu_l|\|w_l\|) = \sum_{1 \le k \le m \atop 1 \le l \le n} | \lambda_k \mu_l| \|v_k \otimes w_l\|_{V_h \otimes_{\max} W_h}.$$ It follows that $$\| v \otimes w \|_{V \otimes_{\max} W, M} \le \|v\|_M \|w\|_M.$$
\end{proof}

For the definitions of OMIN and OMAX in the following proposition, we refer to \cite[Definition 3.3]{PTT} and \cite[Definition 3.12]{PTT}.

\begin{prop}
For an Archimedean ordered $*$-vector space $V$, we have $$M_n({\rm OMIN}(V))^+ = (\mathbb M_n (\mathbb C) \otimes_{\min} V)^+ \quad \text{and} \quad M_n({\rm OMAX}(V))^+ = (\mathbb M_n (\mathbb C) \otimes_{\max} V)^+.$$
\end{prop}

\begin{proof}
For $f \in S(\mathbb M_n(\mathbb C))$ and $g \in S(V)$, we have $$f([g(v_{ij})]_{(i,j)}) = f(\sum_{i,j=1}^n g(v_{ij})e_{ij}) = \sum_{i,j=1}^n f(e_{ij})g(v_{ij}) = (f \otimes g)(\sum_{i,j=1}^n e_{ij} \otimes v_{ij}).$$ The first identity follows from \cite[Theorem 3.2]{PTT}.
\end{proof}

Hereafter, we list the complex versions of the statements in Sections 3 and 4. Their proofs are similar to those of real cases, or the restriction and the complexification \cite[Remark 3.14]{PT} enable us to reduce them to the real cases. Hence, most proofs will be omitted.

\begin{prop}
Suppose that $V, W$ and $Z$ are Archimedean ordered $*$-vector spaces and $\Phi : V \times W \to Z$ is a bilinear map such that $\Phi (v,w) \in Z^+$ for all $v \in V^+$ and $w \in W^+$. Then, there exists a unique positive linear map $\tilde \Phi : V \otimes_{\max} W \to Z$ such that $\Phi (v,w) = \tilde \Phi (v \otimes w)$.
\end{prop}

\begin{prop}
Suppose that $S : V_1 \to V_2$ and $T : W_1 \to W_2$ are unital positive linear maps for Archimedean ordered $*$-vector spaces $V_1, V_2, W_1, W_2$. Then \begin{enumerate} \item $S \otimes T : V_1 \otimes_{\min} W_1 \to V_2 \otimes_{\min} W_2$ is a unital positive linear map. \item $S \otimes T : V_1 \otimes_{\max} W_1 \to V_2 \otimes_{\max} W_2$ is a unital positive linear map.
\end{enumerate}
\end{prop}

\begin{defn}
Suppose that $T : V \to W$ is a unital positive surjective linear map for Archimedean ordered $*$-vector spaces $V$ and $W$. We call $T : V \to W$ {\sl an order quotient map} if for any $w$ in $W^+$ and $\varepsilon>0$, we can find an element $v$ in $V$ such that $$v + \varepsilon e_V \in V^+ \quad \text{and} \quad T(v) = w.$$
\end{defn}

\begin{prop}
Suppose that $T : V \to W$ is a unital positive surjective linear map for Archimedean ordered $*$-vector spaces $V$ and $W$. Then $T : V \to W$ is an order quotient map if and only if $\tilde{T} : V \slash \ker T \to W$ is an order isomorphism.
\end{prop}

\begin{prop}
Suppose that $T : V \to W$ is a unital positive linear map for Archimedean ordered $*$-vector spaces $V$ and $W$. Then \begin{enumerate} \item $T : V \to W$ is an order embedding if and only if it is an isometry with respect to the minimal order norms, and \item $T : V \to W$ is an order quotient map if it is a quotient map with respect to the order norms.
\end{enumerate}
\end{prop}

\begin{proof}
(1) This follows from \cite[Theorem 4.22]{PT}.

(2) In the proof of Proposition \ref{order norm}, we consider the Hermitian lifting ${1 \over 2}(v + v^*)$.
\end{proof}

\begin{thm}
(1) For Archimedean ordered $*$-vector spaces $V_1, V_2, W$ and a unital order embedding $\iota : V_1 \to V_2$, the linear map $\iota \otimes {\rm id}_W : V_1 \otimes_{\min} W \to V_2 \otimes_{\min} W$ is a unital order embedding.

(2) For Archimedean ordered $*$-vector spaces $V_1, V_2, W$ and an order quotient map $Q : V_1 \to V_2$, the linear map $Q \otimes {\rm id}_W : V_1 \otimes_{\max} W \to V_2 \otimes_{\max} W$ is an order quotient map.
\end{thm}

\begin{proof}
(1) Combining \cite[Corollary 2.15]{PT} with \cite[Proposition 3.11]{PT}, we obtain a Hahn-Banach type theorem for a state on an Archimedean ordered $*$-vector space.
\end{proof}

\begin{defn}
An Archimedean ordered $*$-vector space $V$ is called {\sl nuclear} if the identity $$V \otimes_{\min} W = V \otimes_{\max} W$$ holds for any Archimedean ordered $*$-vector space $W$.
\end{defn}

\begin{prop}
An Archimedean ordered $*$-vector space $V$ is nuclear if and only if the Archimedean ordered space $V_h$ is nuclear.
\end{prop}

\begin{proof}
Let $W$ be an Archimedean ordered space and $W^{\mathbb C} := W \oplus i W$ be its complexification \cite[Remark 3.14]{PT}. The conclusion follows from $$(V_h \otimes_{\min} W)^+ = (V \otimes_{\min} W^{\mathbb C})^+ \quad \text{and} \quad (V_h \otimes_{\max} W)^+ = (V \otimes_{\max} W^{\mathbb C})^+.$$
\end{proof}

\begin{thm}
Suppose that $\Phi : V \to W$ is a unital positive map for Archimedean ordered $*$-vector spaces $V$ and $W$. The
following are equivalent:
\begin{enumerate}
\item[(i)] the map $${\rm id}_A \otimes \Phi : A \otimes_{\min} V \to A \otimes_{\max} W$$ is positive for any Archimedean ordered $*$-vector space $A$; \item[(ii)] the map $${\rm id}_E \otimes \Phi : E \otimes_{\min} V \to E \otimes_{\max} W$$ is positive for any finite dimensional Archimedean ordered $*$-vector space $E$; \item[(iii)] there exist nets of unital positive maps $\varphi_\lambda : V \to \ell^\infty_{n_\lambda}(\mathbb C)$ and $\psi_\lambda : \ell^\infty_{n_\lambda}(\mathbb C) \to W$ such that $\psi_\lambda \circ \varphi_\lambda$ converges to the map $\Phi$ in the point-norm topology.
\end{enumerate}
\end{thm}

\begin{proof}
$\rm (ii) \Rightarrow (iii).$ The map $${\rm id}_E \otimes \Phi|_{V_h} : E \otimes_{\min} V_h \to E \otimes_{\max} W_h$$ is positive for any finite dimensional Archimedean ordered space $E$. By Theorem \ref{main1}, there exist nets of unital positive maps $\varphi_\lambda : V_h \to \ell^\infty_{n_\lambda}(\mathbb R)$ and $\psi_\lambda : \ell^\infty_{n_\lambda}(\mathbb R) \to W_h$ such that $\psi_\lambda \circ \varphi_\lambda$ converges to $\Phi|_{V_h}$ in the point-norm topology. Let $\tilde{\varphi}_\lambda : V \to \ell^\infty_{n_\lambda}(\mathbb C)$ and $\tilde{\psi}_\lambda : \ell^\infty_{n_\lambda}(\mathbb C) \to W$ be the complexifications of $\varphi$ and $\psi$, respectively. Then, $\tilde{\psi}_\lambda \circ \tilde{\varphi}_\lambda$ converges to the map $\Phi$ in the point-norm topology.
\end{proof}

\begin{thm}
Let $V$ be an Archimedean ordered $*$-vector space. The following are equivalent:
\begin{enumerate}
\item[(i)] $V$ is nuclear; \item[(ii)] we have $$V \otimes_{\min} E = V \otimes_{\max} E$$ for any finite dimensional Archimedean ordered $*$-vector space $E$; \item[(iii)] there exist nets of unital positive maps $\varphi_\lambda : V \to \ell^\infty_{n_\lambda}(\mathbb C)$ and $\psi_\lambda : \ell^\infty_{n_\lambda}(\mathbb C) \to V$ such that $\psi_\lambda \circ \varphi_\lambda$ converges to ${\rm id}_V$ in the point-norm topology; \item[(iv)] there exist nets of weak$^*$-continuous unital positive maps $\varphi_\lambda : V^{**} \to \ell^\infty_{n_\lambda}(\mathbb C)$ and $\psi_\lambda : \ell^\infty_{n_\lambda}(\mathbb C) \to V^{**}$ such that $\psi_\lambda \circ \varphi_\lambda$ converges to ${\rm id}_{V^{**}}$ in the point-weak$^*$ topology; \item[(v)] the second dual $V^{**}$ is injective; \item[(vi)] given $V \subset V^{**} \subset C(K)$ with a compact set $K$, there exists a unital positive map $\Phi : C(K) \to V^{**}$ such that $\Phi|_V = {\rm id}_V$; \item[(vii)] given both any Archimedean ordered $*$-vector space $W_1$ containing $V$ and an Archimedean ordered $*$-vector space $W_2$, the inclusion $V \otimes_{\max} W_2 \subset W_1 \otimes_{\max} W_2$ is an order embedding; \item[(viii)] given both any Archimedean ordered $*$-vector space $W$ containing $V$ and a finite dimensional Archimedean ordered $*$-vector space $E$, the inclusion $V \otimes_{\max} E \subset W \otimes_{\max} E$ is an order embedding; \item[(ix)] given both any Archimedean ordered $*$-vector spaces $W_1 \subset W_2$, the inclusion $V \otimes_{\max} W_1 \subset V \otimes_{\max} W_2$ is an order embedding; \item[(x)] the state space $S(V_h)$ is a Choquet simplex.
\end{enumerate}
\end{thm}

\bigskip

{\bf Acknowledgments}

The author is grateful to the referee for careful reading and bringing his attention to Ref. \cite{El,F,GL,BR,GK}.



\begin{thebibliography}{AAA}

\bibitem[A]{A} E.M. Alfsen, \textit{Compact convex sets and boundary integrals}, Ergebnisse der Mathematik und ihrer Grenzgebiete, Band \textbf{57}. Springer-Verlag, New York-Heidelberg, 1971.

\bibitem[BR]{BR} G.J.H.M. Buskes and A.C.M. van Rooij, {The Archimedean l-group tensor product}, Order {\bf 10} (1993), no. 1, 93-–102.

\bibitem[Ef]{Ef} E.G. Effros, \textit{Injectives and Tensor Products for Convex Sets and $C^*$-Algebras}, NATO Advanced Study Institute, University College of Swansea, 1972.

\bibitem[El]{El} A.J. Ellis, \textit{Linear operators in partially ordered normed vector spaces}, J. London Math. Soc. {\bf 41} (1966) 323-–332.

\bibitem[EOR]{EOR} E.G. Effros, N. Ozawa and Z.-J. Ruan, {\it On injectivity and nuclearity for operator spaces}, Duke Math. J. {\bf 110} (2001) 489--521.

\bibitem[F]{F} D.H. Fremlin, {\it Tensor products of Archimedean vector lattices}, Amer. J. Math. {\bf 94} (1972), 777-–798.

\bibitem[GK]{GK} O. van Gaans and A. Kalauch, {\it Tensor products of Archimedean partially ordered vector spaces}, Positivity {\bf 14} (2010), no. 4, 705–-714.

\bibitem[GL]{GL} J.J. Grobler and C.C.A. Labuschagne, {\it The tensor product of Archimedean ordered vector spaces},
Math. Proc. Cambridge Philos. Soc. {\bf 104} (1988), no. 2, 331-–345.


\bibitem[HP]{HP} K.H. Han and V.I. Paulsen, \textit{An approximation theorem for nuclear operator systems}, J. Funct. Anal. \textbf{261} (2011) 999--1009.

\bibitem[Ka]{Ka} R. V. Kadison, \textit{A representation theory for commutative topological algebra}. Mem. Amer. Math. Soc. \textbf{1951} (1951) no. 7.

\bibitem[KPTT1]{KPTT1} A. Kavruk, V.I. Paulsen, I.G. Todorov and M. Tomforde, \textit{Tensor products of operator systems}, J. Funct. Anal. \textbf{261} (2011) no.2 267--299.

\bibitem[KPTT2]{KPTT2} A. Kavruk, V.I. Paulsen, I.G. Todorov and M. Tomforde, \textit{Quotients, exactness and WEP in the operator systems category}, Adv. Math. \textbf{235} (2013) 321–-360.

\bibitem[L]{L} H. Lin, \textit{An introduction to the classification of amenable $C\sp *$-algebras}, World Scientific Publishing Co., Inc., River Edge, NJ, 2001.

\bibitem[NP]{NP} I. Namioka and R. R. Phelps, \textit{Tensor products of compact convex sets}, Pacific J. Math. \textbf{31} (1969) 469–-480.

\bibitem[PT]{PT} V.I. Paulsen and M. Tomforde, \textit{Vector spaces with an order unit},
Indiana Univ. Math. J. \textbf{58} (2009) no.3 1319--1359.

\bibitem[PTT]{PTT} V.I. Paulsen, I.G. Todorov and M. Tomforde, \textit{Operator system structures on ordered spaces}, Proc. London Math. Soc. (3) \textbf{102} (2011) 25--49.

\end{thebibliography}
\end{document}